\documentclass[10pt,bezier]{article}
\usepackage{amsmath,amssymb,amsthm,graphics,latexsym,amsfonts,amsfonts}
\usepackage{fancyhdr}
\usepackage{algorithm, algpseudocode}
\usepackage[nooneline,flushleft]{caption2}
\usepackage{indentfirst}

\newtheorem{lemma}{Lemma}[section]
\newtheorem{theorem}[lemma]{Theorem}
\newtheorem{defi}[lemma]{Definition}
\newtheorem{corollary}[lemma]{Corollary}
\newtheorem{prop}[lemma]{Proposition }
\theoremstyle{plain}

\title{ Maximally connected and super arc-connected Bi-Cayley digraphs
\footnote{The research is supported by NSFC (No.10671165) and SRFDP
(No.20050755001).}}
\author{  Yuhu Liu
\footnote{Corresponding author. E-mail:
lyh1120110001@mail.nankai.edu.cn (Y.H.Liu),
mjx@xju.edu.cn  (J.Meng).} \quad  Jixiang Meng \\
{\small\bf College of Mathematics and System Sciences, Xinjiang
University,}\\ {\small\bf Urumqi, Xinjiang, 830046, People¡¯s
Republic of China.}}
\date{}

\begin{document}
\maketitle

\noindent{\bf Abstract} %\baselineskip 0.8cm
Let $X=(V, E)$ be a digraph. $X$ is \emph{maximally connected}, if
$\kappa(X)=\delta(X)$.  $X$ is \emph{maximally arc-connected}, if
$\lambda(X)=\delta(X)$. And $X$ is \emph{super arc-connected}, if
every minimum  arc-cut of $X$ is either the set of inarcs of some
vertex or the set of outarcs of some vertex. In this paper, we will
prove that the strongly connected \emph{Bi-Cayley digraphs} are
maximally connected and maximally arc-connected, and the most of
strongly connected Bi-Cayley digraphs are super arc-connected.


\section {\large Introduction}
A \emph{digraph} is a pair $X=(V, E)$, where $V$ is a finite set and
$E$ is an irreflexive relation on $V$. Thus $E$ is a set of ordered
pairs $(u,v)\in V\times V$ such that $u\neq v$. The elements of $V$
are called the {\em vertices} or \emph{nodes} of $X$ and the
elements of $E$ are called the \emph{arcs} of $X$. Arc $(u,v)$ is
said to be an \emph{inarc} of $v$ and an \emph{outarc} of $u$; we
also say that $(u, v)$ originates at $u$ and terminates at $v$. If
$u$ is a vertex of $X$, then the \emph{outdegree} of $u$ in $X$ is
the number \emph{d}$_{X}^{+}$($u$) of arcs of $X$ originating at $u$
and the \emph{indegree} of $u$  in $X$ is the number
\emph{d}$_{X}^{-}$($u$) of arcs of $X$ terminating at $u$. The
minimum outdegree of $X$ is
$\delta^{+}$(X)=min{\{\emph{d}$_{X}^{+}$($u$) $|$ $u\in V$\}} and
the minimum indegree of $X$ is
$\delta^{-}$($X$)=min{\{\emph{d}$_{X}^{-}$($u$) $|$ $u\in V$\}}. We
denote by $\delta(X)$ the minimum of $\delta^{+}$($X$) and
$\delta^{-}$($X$).

The \emph{reverse} digraph of digraph $X=(V, E)$ is the digraph
$X^{(r)}=(V,\{(v,u)$ $|$ $(u,v)\in E\})$. Digraph $X=(V,E)$ is
\emph{symmetric} if $E$=$E^{(r)}$ and is \emph{antisymmetric} if
$E\bigcap E^{(r)}$=$\varnothing$. An (undirected) graph is a pair
$X=(V, E)$ where $V$ is a finite set and $E$ is a collection of
two-element subsets of $V$. We will in general identify an
undirected graph $X=(V, E)$ with the symmetric digraph $X_s=(V,
E_s)$ where $E_s$=$\{(u, v)| \{u, v\}\in E\} \cup \{(v, u)| \{u,
v\}\in E\}$. A digraph with exactly one vertex (and therefore no
arcs) is called a \emph{trivial} digraph. We denote by $K_n^*$ the
digraph with vertices the integers from 1 to $n$ and arcs all pairs
$(i, j)$ of such integers with $i \neq j$. A digraph isomorphic to
$K_n^*$ is said to be a \emph{complete symmetric digraph}.

For a digraph $X=(V,E)$ and a subset $A$ of $V$, we can get a
subdigraph $X[A]$ of $X$ whose vertex set is $A$ and whose arc-set
consists of all arcs of $X$ which have both ends in $A$. And we call
the subdigraph $X[A]$ is an induced subdigraph of $X$.

\begin{defi} Let $G$ be a group and $T_{0}$, $T_{1}\subseteq
 G$. We define the Bi-Cayley digraph $X= BD(G, T_{0}, T_{1})$ to
be the bipartite digraph with vertex set $G\times\{0, 1\}$ and arc
set \{(($g$, 0), ($t_{0}\cdot$$g$, 1)), (($t_{1}\cdot$$g$, 1), ($g$,
0)) $|$ $g$$\in$$G$, $t_{0}\in$$T_{0}$, $t_{1}\in$$T_{1}$ \} .
\end{defi}
By definition we observe that  $d_{X}^{+}$((g,0))=$|T_{0}|$,
$d_{X}^{-}((g,0))=|T_{1}|$, $d_{X}^{+}((g,1))$\\$=|T_{1}|$,
$d_{X}^{-}((g,1))=|T_{0}|$, for any $g\in G$.

In this paper, we always denote $X_0=G\times \{0\}$ and
$X_1=G\times\{1\}$. Some new results on the Bi-Cayley graph are
referred to \cite{Gao,Kov¨¢cs, Lu, Luo}. In this paper, we will
consider Bi-Cayley digraphs. Denote $R(G)=\{R(a)|R(a):
(g,i)\rightarrow(ga,i),$ for $a, g\in G$ and i=0, 1$\}$. We will get
the following proposition.

\begin{prop} Let $X=BD(G, T_{0}, T_{1})$, then

$(1)$ $R(G)\leqslant$Aut(X), furthermore Aut(X) acts transitively
both on $X_{0}$ and $X_{1}$.

$(2)$ X is strongly connected if and only if $|T_0|\geq 1$,
$|T_1|\geq1$ and $G=<T_1^{-1}T_0>$  .
\end{prop}
\begin{proof} $(1)$ For any $R(a)\in R(G)$ and $((g_{1},0),(g_{2},1))\in
E(X)$, there exists some $t_{0}\in T_{0}$ such that
$g_{2}=t_{0}g_{1}$, then $g_{2}a=t_{0}g_{1}a$. Thus
$((g_{1},0),(g_{2},1))^{R(a)}\\=((g_{1}a, 0), (g_{2}a, 1))\in E(X)$.
Similarly, if $((g_{2},1),(g_{1},0))\in E(X)$, then
$((g_{2},1),(g_{1},0))^{R(a)}\in E(X)$.  So R(a) is an automorphism
of the Bi-Cayley digraph X, thus $R(G)\leq \emph{Aut(X)}$. Since
$(g_{1}, i)^{R(g_{1}^{-1}g_{2})}=(g_{2},i)$ for any $g_{1},g_{2}\in
G$, \emph{Aut(X)} acts transitively both on $X_{0}$ and $X_{1}$.\\
(2) If $X$ is strongly connected, then $|T_0|\geq 1$, $|T_1|\geq1$
and there exists a directed path from $(1_G,0)$ to $(g,0)$ for any
$g\in G$. Thus there exists an integer $n$, $t_0^{(i)}\in T_0$ and
$t_1^{(i)}\in T_1(1\leq i\leq n)$ such that $1_G\rightarrow
t_0^{(1)}\rightarrow
(t_1^{(1)})^{-1}t_0^{(1)}\rightarrow\cdot\cdot\cdot\rightarrow
(t_1^{(n)})^{-1}t_0^{(n)}
\cdot\cdot\cdot(t_1^{(2)})^{-1}t_0^{(2)}(t_1^{(1)})^{-1}t_0^{(1)}=
g$, that is $G=<T_1^{-1}T_0>$. On the other hand,  for any $h$,
$g\in G$, $h^{-1}g$ is in $ G=<T_1^{-1}T_0>$ if and only if it can
be written as a product of elements of $T_1^{-1}T_0 \cup
(T_1^{-1}T_0)^{-1}$. Thus we can easily know there exists a path
from $(g,i)$ to $(h,i)$. And since $|T_0|\geq1$ and $|T_1|\geq1$,
$(g, 0)$ has both outarcs and inarcs for any $g\in G$. So $X$ is
strongly connected.
\end{proof}

\section{Connectivity }

Let $X=(V,E)$ be a strongly connected digraph. An \emph{arc
disconnecting set} of $X$ is a subset $W$ of $E$ such that
$X$$\setminus$W=(V, E$\setminus$W) is not strongly connected. An arc
disconnecting set is \emph{minimal} if no proper subset of $W$ is an
arc disconnecting set of $X$ and is a \emph{minimum arc
disconnecting set} if no other arc disconnecting set has smaller
cardinality than $W$. The \emph{arc connectivity $\lambda$(X)} of a
nontrivial digraph $X$ is the cardinality of a minimum arc
disconnecting set of $X$.

The \emph{positive arc neighborhood} of a subset A of V is the set
$\omega_{X}^ {+}($A$)$ of all arcs which initiate at a vertex of A
and terminate at a vertex of $V$$\setminus$$A$. The \emph{negative
neighborhood} of subset $A$ of $V$ is the set $\omega_{X}^ {-}($A$)$
of all arcs which initiate in $V$$\setminus$$A$ and terminate in
$A$. Thus $\omega_{X}^ {-}($A$)$=$\omega_{X}^ {+}($V$ \setminus
$A$)$. Arc neighborhoods of proper, nonempty subsets of $V$, often
called \emph{arc-cuts}, are clearly arc disconnecting sets. Thus for
any proper, nonempty subset $A$ of $V$, $|\omega^+(A)|\geq
\lambda(X)$. If we consider the cases where $A$ consists of a single
vertex or the complement of a single vertex, we easily see that
$\lambda(X)\leq\delta(X)$.

 A nonempty subset $A$ of $V$ is called a \emph{positive(respectively,
negative) arc fragment} of $X$ if $|\omega^+(A)|=\lambda(X)$
(respectively, $|\omega^-(A)|=\lambda(X)$). An arc fragment $A$ with
$2\leq |A|\leq|V(X)|-2$ is called a \emph{strict arc fragment } of
$X$. An arc fragment of minimum cardinality is called
\emph{$\lambda$-atom} of $X$ and a strict arc fragment of least
possible cardinality is called a \emph{$\lambda$-superatom} of $X$.
Note that a $\lambda$-atom(respectively, $\lambda$-superatom) may be
either a positive arc fragment or a negative arc fragment or both. A
$\lambda-$atom which is a positive(respectively, negative) arc
fragment is called a \emph{positive(respectively, negative)
$\lambda$-atom} and a $\lambda-$superatom which is a
positive(respectively, negative) arc fragment is called a
\emph{positive(respectively, negative) $\lambda$-surperatom}.

A \emph{vertex disconnecting set} of $X$ is a subset $F$ of $V(X)$
such that $X$$\setminus$F is either trivial or is not strongly
connected. We often call $F$ a \emph{vertex-cut}. The
\emph{connectivity $\kappa$(X)} of a nontrivial digraph $X$ is the
cardinality of a minimum vertex disconnecting set of $X$.

The \emph{positive neighborhood} of a subset $F$ of $V$ is the set
$N^{+}(F)$ of all vertices of $V\setminus F$ which are targets of
arcs initiating at a vertex of $F$. The \emph{positive closure}
$C^{+}(F)$ of $F$ is the union of $F$ and $N^+(F)$. The
\emph{negative neighborhood} of subset $F$ of $V$ is the set
$N^{-}(F)$ of all vertices of $V\setminus F$ which are the initial
vertices of arcs which terminate at a vertex of $F$. The
\emph{negative closure} $C^{-}(F)$ of $F$ is the union of $F$ and
$N^-(F)$.

If $F$ is a nonempty subset of $V$ with $C^+(F)\neq V$, then the
positive neighborhood of $F$ is clearly a vertex disconnecting set
for $X$. Thus for each such set $F$, $|N^+(F)|\geq\kappa(X)$. If we
consider the cases where $F$ consists of a single vertex or the
complement of a single vertex, we easily see that
$\kappa(X)\leq\delta(X)$.

A nonempty subset $F$ of $V$ is called a
\emph{positive}(respectively, \emph{negative}) \emph{fragment} of
$X$ if $|N^+(F)|=\kappa(X)$ and $C^+(F)\neq V$ (respectively,
$|N^-(F)|=\kappa(X)$ and $C^-(F)\neq V$). A fragment  of minimum
cardinality is called \emph{atom}. Note that an atom may be either a
positive fragment or a negative fragment or both. A atom which is a
positive(respectively, negative) fragment is called a
\emph{positive(respectively, negative) atom }.

A digraph $X$ is \emph{maximally arc connected}(respectively,
\emph{maximally connected}), or more simply,
\emph{max-$\lambda$}(respectively, \emph{max-$\kappa$}), if
$\kappa(X)=\delta(X)$(respec-\\tively, $\lambda(X)=\delta(X)$). And
$X$ is \emph{super arc connected}, or more simply,
\emph{super-$\lambda$} if every minimum arc-cut of $X$ is either the
set of inarcs of some vertex or the set of outarcs of some vertex.
The relationship of $\lambda(X)$ and $\kappa(X)$ is well known:
$\kappa(X)\leq\lambda(X)\leq\delta(X)$. So if $\kappa(X)=\delta(X)$,
then $\lambda(X)=\delta(X)$. In the following of this section we
will try to prove that $\kappa(X)=\delta(X)$ for Bi-Cayley digraphs.

A desirable property one wishes any type of atom to have is that, if
nontrivial, they form \emph{imprimitive blocks} for the automorphism
group of the digraph. To be precise, an \emph{imprimitive block} for
a group $\Phi$ of permutations of a set $T$ is a proper, nontrivial
subset $A$ of $T$ such that if $\varphi\in\Phi$ then either
$\varphi(A)=A$ or $\varphi(A)\cap A=\emptyset$. In the following
proposition Hamidoune has proved that the positive(respectively,
negative) atoms of $X$ are imprimitive blocks of $X$. The following
proposition indicates why imprimitivity is so useful.
\begin{prop}\cite{HamTin} Let $X=(V,E)$ be a
graph or digraph and let $Y$ be the subgraph or subdigraph induced
by an imprimitive block $A$ of $X$. Then\\1. If $X$ is
vertex-transitive then so is $Y$;\\2. If $X$ is a strongly connected
arc-transitive digraph or a connected edge-transitive graph and $A$
is a proper subset of $V$, then $A$ is an independent subset of $X$.
\end{prop}

\begin{prop} \cite{Hamidoune} Let $X=(V, E)$ be a strongly connected
digraph which is not a complete symmetric digraph and let $A$ be a
positive (respectively, negative) atom of $X$. If $B$ is a
positive(respectively, negative)fragment of $X$ with $A\cap
B\neq\emptyset$, then $A\subset B$.
\end{prop}

\begin{prop} Let $X$ be a strongly digraph with
$\kappa(X)<\delta(X)$, and $A$ be an atom of $X$, then $X[A]$ is
strongly connected.
\end{prop}

Clearly if $X=BD(G,T_0,T_1)$ is a strongly connected Bi-Cayley
digraph with $\kappa(X)<\delta(X)$ and $A$ is an atom of $X$, then
$A_i=A\cap X_i\neq\varnothing$ for $i=0, 1$.

\begin{lemma} Let $X=BD(G, T_0, T_1)$ be a strongly connected Bi-Cayley digraph with $\kappa(X)<\delta(X)$. If A is an
atom of X, then
\\(1) $V(X)$ is a disjoint union of distinct
positive(or, negative) atoms of $X$;
\\(2) Let $Y=X[A]$, then $Aut(Y)$ acts transitively both on $A_0$
and $A_1$;\\(3) If  $(1,i)\in A_i=H_i\times\{i\}$, then $H_i$ is the
subgroup of G for i=0,1;\\ (4) $|A_0|=|A_1|$.
\end{lemma}
\begin{proof} (1) and (2) follow from the results that the
distinct positive(negative) atoms are disjoint and $Aut(X)$ acts
transitively both on $X_0$ and $X_1$.
\\(3) For any $g\in H_0$, $Ag$ is also a positive
atom  since R(g)$\in$ Aut(X). And $g\in A\cap Ag$, then we get that
$A=Ag$, thus $A_0g=A_0$ and $A_1g=A_1$. The former equality means
that $H_0$ is a subgroup of G.\\(4) From proposition 1.2(1) and
proposition 2.2, we can get $V(X)=\cup_{i=1}^{k}\varphi_{i}(A)$
where $\varphi_{i}\in \emph{Aut(X)}$ such that $\varphi_{i}(A)\cap
\varphi_{j}(A)=\emptyset$ if $i\neq j$, then
$X_{i}=\cup_{i=1}^{k}\varphi_{i}(A_{i})$. Since $|X_{0}|=|X_{1}|$,
we have $|A_{0}|=|A_{1}|$.
\end{proof}

From the proof of lemma 2.4, $Y=X[A]$ has the property that
$d_Y^+((g_i,0))\\=d_Y^+((g_j,0))$ and $d_Y^-((g_i,
0))=d_Y^-((g_j,0))$ for any vertices $(g_i, 0)$, $(g_j, 0)$$\in
A_0$. And if $(1,0)\in A_0$,then $A_1g=A_1$ is right for any $g\in
H_0$, so $H_1H_0=H_1$. It means $H_1$ is a left coset of $H_0$ since
$|H_0|=|H_1|$. We have the following lemma.
\begin{lemma} Let $X=BD(G, T_0, T_1)$ be a strongly connected
Bi-Cayley digraph with $\kappa(X)<\delta(X)$, and $A$ be a positive
atom. Let $A_0=\{g_1, g_2, ... , g_m\}\\ \times\{0\}=H_0\times\{0\}$
and  $A_1=\{g_1^{'}, g_2^{'},
 ... , g_m^{'}\} \times\{1\}=H_1\times\{1\}$. Then\\(1) If
$t_ig_j\in H_1$ for some $t_i\in T_i$ (i=0, 1) and some some
$j(1\leq
 j\leq m)$, then $t_ig_k\in H_1$ for any $k(1\leq k\leq m)$;
 \\(2) If
 $t_1^{-1}g_j^{'}\in H_0$ for some $t_1\in T_1$ and some $j(1\leq
 j\leq m)$, then $t_1^{-1}g_k^{'}\in H_0$ for any $k(1\leq k\leq
 m)$.
\end{lemma}

\begin{proof}(1) Assume $(1, 0)\in A_0$, then $H_1H_0=H_1$. If $t_ig_j\in H_1$, then $t_ig_jH_0=t_iH_0=H_1$. It means $t_ig_k\in H_1$ for any $k(1\leq k\leq m)$. \\(2) Similarly, assume $(1, 0)\in A_0$, then $H_1H_0=H_1$. If $t_1^{-1}g_j^{'}\in H_0$, then $g_j^{'}\in t_1H_0$. So
$H_1=t_1H_0$. it means that $t_1^{-1}g_k^{'}\in H_0$ for any
$k(1\leq k\leq
 m)$.
\end{proof}

\begin{theorem}
Let $X=BD(G, T_0, T_1)$ be a strongly connected Bi-Cayley digraph,
then $\kappa(X)=\delta(X)$
\end{theorem}
\begin{proof} Suppose $X$ is not max-$\kappa$. Without loss of
generality, assume that $A=A_0\cup A_1$ is a positive atom. Denote
$A_0=H_0\times \{0\}$ and $A_1=H_1\times \{1\}$.  If
$|N^+(A_0)\setminus A_1|\neq 0$, then by lemma 2.5  we have
$|N^+(A_0)\setminus A_1|\geq |H_0|$. Thus
$|N^+(A)|=|N^+(A_0)\setminus A_1|+|N^+(A_1)\setminus
A_0|=|N^+(A_0)\setminus A_1|+|\{T_1^{-1}H_1\setminus
H_0\}\times\{0\}|\geq |H_0|+|T_1^{-1}\setminus H_0|\geq
|T_1^{-1}|\geq \delta (X)$, a contradiction. If $|N^+(A_1)\setminus
A_0|\neq 0$, then by lemma 2.5  we have $|N^+(A_1)\setminus A_0|\geq
|H_1|$. Thus $|N^+(A)|=|N^+(A_0)\setminus A_1|+|N^+(A_1)\setminus
A_0|=|\{T_0H_0\setminus H_1\}\times\{1\}|+|N^+(A_0)\setminus
A_1|\geq |T_0\setminus H_1|+|H_1|\geq |T_0|\geq \delta (X)$, a
contradiction. Therefore $N^+(A)=\emptyset$, it is a contradiction .
\end{proof}

\begin{corollary}
Let $X=BD(G, T_0, T_1)$ be a strongly connected Bi-Cayley digraph,
then $\kappa(X)=\lambda(X)=\delta(X)$.
\end{corollary}

\section{Super arc-connectivity}

A \emph{weak path} of a digraph $X$ is a sequence $u_0,...,u_r$ of
distinct vertices such that for $i=1,...,r,$ either $(u_{i-1},u_i)$
or $(u_i,u_{i-1})$ is an arc of $X$. A directed graph is
\emph{weakly connected} if any two vertices can be joined by a weak
path.

\begin{prop} Let $X=BD(G,T_0,T_1)$ be a strongly connected Bi-Cayley digraph and $A$ be a $\lambda-$superatom, then \\(1) $Y=X[A]$ is weakly
connected;\\(2) $|A|\geq \delta(X)$.
\end{prop}
\begin{proof} Suppose $A$ is a positive $\lambda-$superatom.\\(1) If $|A|=2$, then we obtain that $A$ is not an
independent set since $|N^+(A)|=\delta(X)$ and $N^+(u)\neq 0$ for
any $u\in V(X)$. Now assume $|A|\geq3$. If $Y=X[A]$ is not weakly
connected, we can get a $\lambda-$superatom with cardinality less
than $A$, a contradiction.
\\(2)
$\lambda(X)=|\omega_X^{+}(A)|\geq|A|(\delta(X)-(|A|-1))=|A|(\delta(X)-|A|+1)$,
\\we can verify that $\lambda(X)>\delta(X)$ when
$2\leq|A|<\delta(X)$,  a contradiction.
\end{proof}

Any digraph with $d^+(x)=d^-(x)$ for every vertex $x$ of $X$ is said
to be a \emph{balanced} digraph.

\begin{prop}\cite{HamTin} Let $X=(V,E)$ be a strongly
connected, balanced digraph and let $A$ and $B$ be arc fragments of
X such that $A\nsubseteq B$ and $B\nsubseteq A$. If $A\cap
B\neq\varnothing$ and $A\cup B\neq V$, then each of the sets $A\cap
B$, $A\cup B$, $A\setminus B$ and $B\setminus A$ is an  arc
fragments of  $X$.
\end{prop}
\begin{theorem}\cite{HamTin}  Let $X=(V,E)$ be a strongly
connected balanced digraph which is not a symmetric cycle, is not
super arc-connected and has $\delta(X)\geq2$. If $\delta(X)>2$ or
$X$ is vertex-transitive, then distinct $\lambda-$superatoms of $X$
are vertex disjoint.
\end{theorem}
Similarly, we can also achieve the analogous results.
\begin{prop} Let $X=(V,E)$ be a strongly
connected digraph and let $A$ and $B$ be positive(respectively,
negative) arc fragments of X such that $A\nsubseteq B$ and
$B\nsubseteq A$. If $A\cap B\neq\varnothing$ and $A\cup B\neq V$,
then each of the sets $A\cap B$, $A\cup B$, $A\setminus B$ and
$B\setminus A$ is a positive(respectively, negative) arc fragments
of $X$.
\end{prop}
\begin{theorem} Let $X=(V,E)$ be a strongly
connected digraph which is not a symmetric cycle, is not super
arc-connected and has $\delta(X)\geq2$. If $\delta(X)>2$ or $X$ is
vertex-transitive, then distinct positive(respectively,
negative)$\lambda-$superatoms of $X$ are vertex disjoint.
\end{theorem}
\begin{lemma} Let $X=BD(G,T_0,T_1)$ be strongly connected but
not super$-\lambda$. If $X$ is neither a directed cycle nor a
symmetric cycle , then distinct positive(respectively, negative)
$\lambda-$superatoms of $X$ are vertex disjoint.
\end{lemma}
\begin{proof} Suppose to the contrary that there are distinct
positive $\lambda-$superatoms $A$, $B$ of $X$ with $A\cap B \neq
\emptyset$. By proposition 3.4, each of $A\cap B,A\cup B,A\setminus
B,B\setminus A$ is a positive arc fragment which is a proper subset
of a $\lambda-$superatom. Therefore, each of these sets must have
cardinality 1 so that we may assume $A=\{u,v\}$, $B=\{v,w\}$ with
$u\neq w$. Thus we have
 $d_{X[A]}^+(u)=d_{X[A]}^-(v)\leq1$,
$d_{X[A]}^-(u)=d_{X[A]}^+(v)\leq1$,
$d_{X[B]}^+(v)=d_{X[B]}^-(w)\leq1$ and
$d_{X[B]}^-(v)=d_{X[B]}^+(w)\leq1$. \\\textbf{Case 1} $\delta(X)=1$.

$d_X^+(u)=d_X^+(v)=d_X^+(w)=1$, so $|T_0|=1$, $|T_1|=1$. And because
$X$ is a strongly connected digraph, we can get $X$ is a directed
cycle, a contradiction.\\\textbf{Case 2} $\delta(X)=2$.

$d_X^+(u)=d_X^+(v)=d_X^+(w)=2$. \\Because $X[A]$ and $X[B]$ are
weakly connected and $A$, $B$ and $A\cup B$ are arc fragments, we
can deduce

$|T_0|=|T_1|=2$ and $T_0=T_1$.\\ Because $X$ is strongly connected,
$X$ is a cycle, a contradiction.\\\textbf{Case 3} $\delta(X)\geq3$.

It is true by Theorem 3.5.
\end{proof}
For the rest of the paper we set $A_i = A \cap X_1 = H_i \times
{i}$, i = 0,1. Similarly to Lemma 2.4, we can derive the following
theorem.

\begin{lemma} Let $X=BD(G ,T_0,T_1)$, which is neither a directed cycle nor a symmetric cycle, be strongly connected but not
super$-\lambda$ . Let $A$ be a $\lambda-$superatom of $X$, then\\(1)
$V(X)$ is a disjoint union of distinct positive(negative)
$\lambda-$superatoms;\\(2) Let $Y=X[A]$, then  \emph{Aut($Y$)} acts
transitively both on $A_0$ and $A_1$; \\(3) If $A_i$ contains $(1,
i) (i=0, 1)$, then $H_i$ is a subgroup of $G$;\\(4) $|A_0|=|A_1|$.
\end{lemma}

Similarly as lemma 2.4, we also have $H_1H_0=H_1$ if $(1,0)\in A_0$
and $H_0H_1=H_0$ if $(1,1)\in A_1$. The following proposition is
easy to get.

By a similar argument as lemma 2.5, the following lemma is obtained.
\begin{lemma} Let $X=BD(G ,T_0,T_1)$, which is neither a directed cycle nor a symmetric cycle, be strongly connected but not
super$-\lambda$ . Let $A$ be a $\lambda-$superatom of $X$ and set
$A_0=\{g_1, g_2, g_3, ... , g_m\}\times\{0\}=H_0\times\{0\}$ and
$A_1=\{g_1^{'}, g_2^{'}, g_3^{'}, ... ,
g_m^{'}\}\times\{1\}=H_1\times\{1\}$. Then \\(1) If
 $t_ig_j\in H_1$ for some $t_i\in T_i$ (i=0, 1) and some some $j(1\leq
 j\leq m)$, then $t_ig_k\in H_1$ for any $k(1\leq k\leq m)$;
 \\(2) If
 $t_1^{-1}g_j^{'}\in H_0$ for some $t_1\in T_1$ and some $j(1\leq
 j\leq m)$, then $t_1^{-1}g_k^{'}\in H_0$ for any $k(1\leq k\leq
 m)$.
\end{lemma}
\begin{theorem} Let $X=BD(G,T_0,T_1)$ be strongly connected. If
$X$ is neither a directed cycle nor a symmetric cycle, then $X$ is
not super$-\lambda$ if and only if $X$ satisfies one of the
following conditions:\\(1) There exists a subgroup $H\leq G$ and
there distinct elements $t_0, t_0^{'}, t_0^{''}\in T_0$ such that

$|H|=\delta(X)$, $T_1^{-1}t_0\subseteq H$
,$t_0^{-1}(T_0\setminus\{t_0^{'}\})\subset H$ and
$t_0^{-1}t_0^{'}\notin H .$

or

$|H|=\delta(X)$/2, $T_1^{-1}t_0\subseteq H$,
$t_0^{-1}(T_0\setminus\{t_0^{'},t_0^{''}\})\subset H$ and
$t_0^{-1}t_0^{'},t_0^{-1}t_0^{''}\notin H$ .Where $t_0^{'}\neq t_0$
and $t_0^{''}\neq t_0$\\(2) There exists a subgroup $H\leq G$ and
two distinct elements $t_1,t_1^{'},\in T_1$ and some element $t_0\in
T_0$ such that

$|H|=\delta(X)$, $t_0^{-1}T_0\subset H$,
$(T_1\backslash\{{t_1^{'}}\})^{-1}t_0\subset H$ and
${t_1^{'}}^{-1}t_0\notin H.$

or

$|H|=\delta(X)$/2, $t_0^{-1}T_0\subset H$,
$(T_1\backslash\{t_1,t_1^{'}\})^{-1}t_0\subset H$ and
${t_1}^{-1}t_0, {t_1^{'}}^{-1}t_0 \notin H$.\\(3) There exists a
subgroup $H\leq G$ and two distinct elements $t_0, t_0^{'}\in T_0$
and some element $t_1\in T_1$ such that $|H|=\delta(X)/2$,
$t_0^{-1}(T_0\setminus\{t_0^{'}\})\subset H$, $t_0^{-1}t_0^{'}\notin
H$, $(T_1\backslash\{{t_1}\})^{-1}t_0\subset H$ and
$t_1^{-1}t_0\notin H$, where $t_0^{'}\neq t_0$.

\end{theorem}
\begin{proof}
Necessity. Without loss of generality assume $A$ is a positive
$\lambda-$super-atom of $X$ and $(1, 0)\in A$. From lemma 3.7, $H_0$
is a subgroup of $G$ and  $Y=X[A]$ is a bipartite digraph with
$d_Y^+((g_k,i))$=$d_Y^+((g_l,i))$ and
$d_Y^-((g_k,i))$=$d_Y^-((g_l,i))$ for any vertices $(g_k,i) , (g_l,
i)\in A$(i=0, 1). Furthermore, $d_Y^+((g_j,0))=d_Y^-((g_t,1))$ and
$d_Y^-((g_j,0))=d_Y^+((g_t,1))$. Denote
$d_Y^+((g_j,0))=d_Y^-((g_t,1))=p$,
$d_Y^-((g_j,0))=d_Y^+((g_t,1))=q$. Let $H=H_0$.\\\textbf{Claim:}
There exist at least an element $t_0\in T_0$ such that $H_1=t_0H_0$
if $(1, 0)\in A$.\\\textbf{The proof of the Claim: }If $p=0$, then
$\delta(X)=\lambda(X)=|\omega_X^{+}(A)|=|A_0|(|T_0|-p)+|A_1|(|T_1|-q)=
|A_0||T_0|+|A_1|(|T_1|-q)\geq |T_0|+ |A_1|(|T_1|-q)\geq|T_0|\geq
\delta(X)$. So $|A_0|=|A_1|=1$ and $|T_1|-q=0$. So
$\omega^+(A)=\omega^+(A_0)$, thus $X$ is super$-\lambda$. By a
similar argument we can prove $X$ is super$-\lambda$ when $q=0$. A
contradiction. So $pq\neq 0$.

If $(1,0)\in A_0$, then $H_1=H_1H_0$. And because $p\neq0$, there
exist at least an element $t_0\in T_0$ such that $t_0\in H_1$. Thus
$H_1=t_0H_0$. So the Claim is true.

Since
$\delta(X)=\lambda(X)=|\omega_X^{+}(A)|=|A_0|(|T_0|-p)+|A_1|(|T_1|-q)$,
$|A|\geq\delta(X)$ and $|A_0|=|A_1|$, we have
$|A_0|=|A_1|\geq\delta(X)/2$ and $|T_0|-p+|T_1|-q\leq2$. Now we
consider fine cases.
\\\textbf{Case 1} $|T_0|-p=1$ and $|T_1|-q=0$.
\\(i)
$\lambda(X)=|\omega_X^{+}(A)|=|A_0|=|H_0|=|H|=\delta(X)$, since
$|T_0|-p=1$ and $|T_1|-q=0$.\\(ii) Since $|T_0|-p=1$, there exists
an element $t_0^{'}\in T_0$ such that
$(T_0\setminus\{t_0^{'}\})H_0\subset H_1$ and $t_0^{'}H_0\cap
H_1=\emptyset$. It means $(T_0\setminus\{t_0^{'}\})H_0\subset
t_0H_0$ and $t_0^{'}H_0\cap t_0H_0=\emptyset$, so
$t_0^{-1}(T_0\setminus\{t_0^{'}\})\subset H_0$ and
$t_0^{-1}t_0^{'}\notin H_0$.
\\(iii) since $|T_1|-q=0$, we have that $T_1^{-1}H_1\subseteq H_0$.
It means $T_1^{-1}t_0H_0\subseteq H_0$, so  $T_1^{-1}t_0\subseteq
H_0$.
\\\textbf{Case 2} $|T_0|-p=0$ and $|T_1|-q=1$.\\(i)
$\lambda(X)=|\omega_X^{+}(A)|=|A_1|=|H_1|=|H_0|=|H|=\delta(X)$,
since $|T_0|-p=0$ and $|T_1|-q=1$.\\(ii) Since $|T_0|-p=0$, we have
that $T_0H_0\subseteq H_1$. It means $T_0H_0\subseteq t_0H_0$, so
$t_0^{-1}T_0\subseteq H_0$.
\\(iii) since $|T_1|-q=1$, there exists an element $t_1\in T_1$
such that $(T_1\setminus\{t_1\})^{-1}H_1\subset H_0$ and
$t_1^{-1}H_1\cap H_0=\emptyset$. It means
$(T_1\setminus\{t_1\})^{-1}t_0H_0\subset H_0$ and
$t_1^{-1}t_0H_0\cap H_0=\emptyset$, so
$(T_1\setminus\{t_1\})^{-1}t_0\subset H_0$ and $t_1^{-1}t_0\notin
H_0$.
\\\textbf{Case 3} $|T_0|-p=2$ and $|T_1|-q=0$.
\\It is similar to Case 1, we have\\(i)
$|H|=|H_0|=\delta(X)/2$.\\(ii)
$t_0^{-1}(T_0\setminus\{t_0^{'},t_0^{''}\})\subset H_0$ and
$t_0^{-1}t_0^{'},t_0^{-1}t_0^{''}\notin H_0$ for some
$t_0^{'},t_0^{''}\in T_0$.\\(iii) $T_1^{-1}t_0\subseteq H_0$.
\\\textbf{Case 4} $|T_0|-p=0$ and $|T_1|-q=2$.
\\It is similar to Case 2, we have\\(i) $|H|=|H_1|=\delta(X)/2$.\\(ii)
$t_0^{-1}T_0\subset H_0$.\\(iii)
$(T_1\backslash\{{t_1^{'}},{t_1^{''}}\})^{-1}t_0\subset H_0$ and
 ${t_1^{'}}^{-1}t_0,{t_1^{''}}^{-1}t_0\notin H_0$ for some $t_1^{'},t_1^{''}\in T_1$.
\\\textbf{Case 5} $||T_0|-p=1$ and $|T_1|-q=1$.
\\(i) $\lambda(X)=|\omega_X^{+}(A)|=|A_0|=|H_0|=|H|=\delta(X)/2$,
since $|T_0|-p=1$, and $|T_1|-q=1$.\\(ii) since $|T_0|-p=1$, then
$t_0^{-1}(T_0\setminus\{t_0^{'}\})\subset H_0$ and
$t_0^{-1}t_0^{'}\notin H_0$ for some element $t_0^{'}\in T_0$ and
$t_0^{'}\neq t_0$.
\\(iii) since $|T_1|-q=1$, $(T_1^{-1}\backslash\{{t_1}^{-1}\})t_0\subset
H_0$ and $t_1^{-1}t_0\notin H_0$ for some $t_1\in
T_1$.\\Sufficiency. Set $A=H\times\{0\}\cup(t_0H)\times\{1\}$. Thus
$(1,0)\in A$, $H_0=H$ and $H_1=t_0H_0$.\\(1) If
$t_0^{-1}(T_0\setminus\{t_0^{'}\})\subset H$ and
$t_0^{-1}t_0^{'}\notin H$, then
$t_0^{-1}(T_0\setminus\{t_0^{'}\})H=H$ and $t_0^{'}\notin t_0H$, it
is $(T_0\setminus\{t_0^{'}\})H=t_0H=H_1$ and $t_0^{'}H_0\cap
H_1=\emptyset$. So $|T_0|-p=1$. And if $T_1^{-1}t_0\subseteq H$,
then $T_1^{-1}t_0H\subseteq H$. It is $T_1^{-1}H_1\subseteq H$. So
$|T_1|-q=0$. Associate with the condition $|H|=\delta(X)$, we have
$\lambda(X)=|\omega_X^{+}(A)|=|A_0|=|H|=\delta(X)$. So $A$ is a
$\lambda-$superatom of $X$.\\Similarly, If $|H|=\delta(X)$/2,
$T_1^{-1}t_0\subseteq H$,
$t_0^{-1}(T_0\setminus\{t_0^{'},t_0^{''}\})\subset H$ and
$t_0^{-1}t_0^{'},t_0^{-1}t_0^{''}\notin H$, we can prove $A$ is a
$\lambda-$superatom of $X$.\\(2) If
$(T_1\backslash\{{t_1^{'}}\})^{-1}t_0\subset H$ and
${t_1^{'}}^{-1}t_0\notin H$, then
$(T_1\backslash\{{t_1^{'}}\})^{-1}t_0H\subset H$ and $t_0\notin
t_1H$, it is $(T_1\backslash\{{t_1^{'}}\})^{-1}H_1\subset H$ and
$t_0H\cap t_1H=\emptyset$. So $|T_1|-q=1.$ And if
$t_0^{-1}T_0\subset H$, then $t_0^{-1}T_0H\subset H$, it is
$T_0H\subset t_0H$. So $|T_0|-p=0$. Associate with the condition
$|H|=\delta(X)$, we get
$\lambda(X)=|\omega_X^{+}(A)|=|A_1|=|H_1|=|H_0|=\delta(X)$. So $A$
is a $\lambda-$superatom of $X$. \\Similarly, If $|H|=\delta(X)$/2,
$t_0^{-1}T_0\subset H$,
$(T_1\backslash\{t_1,t_1^{'}\})^{-1}t_0\subset H$ and $t_1^{-1}t_0,
{t_1^{'}}^{-1}t_0 \notin H$, we can prove $A$ is a
$\lambda-$superatom of $X$.\\(3) If
$t_0^{-1}(T_0\setminus\{t_0^{'}\})\subset H$,
$(T_1\backslash\{{t_1}\})^{-1}t_0\subset H$, and $t_0^{-1}t_0^{'},
{t_1}^{-1}t_0\notin H$ then
$t_0^{-1}(T_0\setminus\{t_0^{'}\})H\subset H$,
$(T_1\backslash\{{t_1}\})^{-1}t_0H\subset H$, $t_0^{'}\notin t_0H$
and $t_0\notin t_1H$, it is $(T_0\setminus\{t_0^{'}\})H\subset
t_0H=H_1$, $(T_1\backslash\{{t_1}\})^{-1}H_1\subset H_0$,
$t_0^{'}H\cap t_0H=\emptyset$ and $t_0H\cap t_1H=\emptyset$. So
$|T_0|-p=1$ and $|T_1|-q=1$. Thus associate with the condition
$|H|=\delta(X)/2$, we have
$\lambda(X)=|\omega_X^{+}(A)|=|A_0|+|A_1|=|H_0|+|H_1|=\delta(X)$. So
$A$ is a $\lambda-$superatom of $X$.
\end{proof}

\end{document}